\documentclass[10pt]{amsart}

\usepackage{amssymb,amsmath,euscript}
\usepackage{mathtools}
\usepackage{breqn}
\usepackage{tikz}
\usepackage{tkz-berge}
\usepackage{enumitem}
\usepackage{comment}

\newtheorem{theorem}{Theorem}[section]
\newtheorem{lemma}[theorem]{Lemma}

\newtheorem{proposition}[theorem]{Proposition}

\newtheorem{corollary}[theorem]{Corollary}

\theoremstyle{definition}
\newtheorem{definition}[theorem]{Definition}
\newtheorem{example}[theorem]{Example}

\theoremstyle{remark}
\newtheorem{remark}[theorem]{Remark}
\def\Fq{{\mathbb F}_q}

\def\F{{\mathcal F}}

\def\I{{\mathcal I}}

\def\M{{\mathcal M}}

\newcommand{\kI}{\mathcal{I}}
\newcommand{\kS}{\mathcal{S}}
\newcommand{\kB}{\mathcal{B}}
\newcommand{\kL}{\mathcal{L}}

\newcommand{\cl}{\mathop{\text{cl}}}

\def\LBr1m {\lambda_B(r-1, m)}

\begin{document}
 
	\title{A $q$-analogue of $\Delta$-matroids and related concepts}
	
	\author{Michela Ceria}
\address{Department of Mechanics, Mathematics and Management (DMMM),
\newline \indent
Politecnico di Bari, Italy.}

\email{michela.ceria@poliba.it}
\author{Trygve Johnsen}
\address{Department Mathematics and Statistics,\newline \indent
	UiT: The Arctic University of Norway,\newline \indent
	Hansine Hansens veg 18, N-9019 Tromsø, Norway.}
	\email{Trygve.Johnsen@uit.no}
	\author{Relinde Jurrius}
\address{Faculty of Military Sciences
\newline \indent
Netherlands Defence Academy, the Netherlands}

\email{RPMJ.Jurrius@mindef.nl}	\date{\today}
	
\begin{abstract}
	We define and study $q$-$\Delta$-matroids, and $q$-$g$-matroids. These objects are  analogues, for finite-dimensional vector spaces over finite fields, of  $\Delta$-matroids and $g$-matroids arising from finite sets. We compare axiomatic descriptions with definitions by means of strong maps of $q$-matroids. 

\end{abstract}

\maketitle
\section{Introduction} \label{intro}
Matroids are combinatorial objects that in an axiomatic way enable one to study linear independence in a general way. Nevertheless, or precisely because of this generality,  matroids are used in concrete applications in a multitude of ways, for graphs, error-correcting codes, toric geometry, the study of greedy algorithms, and others.

$\Delta$-matroids are a related concept, developed to supplement matroids, in various applications.
A $\Delta$-matroid is a pair $(E,\mathcal{F})$, where $E=\{1,\cdots,n\},$ and $\mathcal{F}$ is a  non-empty family of subsets of  $E$, such that:

\begin{quote}
For every two sets $X$ and $Y$ in $\mathcal{F}$, and for every element $x$ in their symmetric difference $ X \bigtriangleup Y= (X-Y) \cup (Y-X)$, there exists a $ y\in X\bigtriangleup Y$ such that $ X \bigtriangleup \{x,y\}$ is in the family.
\end{quote}

The non-empty family $\mathcal{F}$ is called the family, or set, of \emph{feasible sets} of the $\Delta$-matroid.

In recent years a need for inventing a $q$-analogue of matroids has arisen, in particular in connection with rank-metric codes. Just like ``usual'' matroids are useful tools for determining properties of linear error-correcting codes with Hamming metric, $q$-matroids can be used to study similar properties of vector rank-metric codes (by some authors called Gabidulin rank-metric codes).

Another arena, where (usual) matroids play an important, but sometimes insufficient role for certain purposes, is the study of graphs embedded in surfaces, or the study of (compact Riemann) surfaces by means of putting graphs or similar objects into them. This  has led to generalizations, not only of graphs, but also of the (cycle) matroids associated to those (sometimes planar) graphs. One has generalized graphs to ribbon graphs, and as a parallel, one has generalized matroids, not to $q$-matroids, but to $\Delta$-matroids, as defined above. As an illustration: If a graph is planar, then its geometrical dual graph (where regions and vertices are interchanged compared to the ``old'' graph) has a cycle matroid which is the dual of the cycle matroid of the ``old'' graph.
The surface we are working on is essentially the compact Riemann surface $S^2$ (when adding a ``compactification point'' to the infinite planar region). If we have a graph embedded in a compact Riemann surface with $g $ holes, for $g \ge 1$, this matching of geometrical and matroidal dualities no longer holds. Therefore, to remedy these difficulties, one invents ribbon graphs and $\Delta$-matroids, and then a new matching of dualities appears.
See for example \cite{CMNR} for an overview.

$\Delta$-matroids are also used in other areas, like constraint satisfaction problems (CSPs), equivariant localization, and tropical geometry. See \cite{FF}, \cite{ELS},  and \cite{R}, respectively.

One should note that the concept of $\Delta$-matroids was introduced by Bouchet in \cite{B3} and treated for example in \cite{B1} and \cite{B2}. See 
\cite[Section 1, Line 8]{B2}, and \cite[Theorem 4.1]{B2} for an identification of a graph-theoretical problem, and a solution to restore the harmony on a different level, respectively.

In this article we investigate whether it is possible to find some kind of natural $q$-analogue of $\Delta$-matroids, in other words a generalization from $q$-matroids to $q$-$\Delta$-matroids, or if one prefers: A superposition of the two changes made from matroids to $q$-matroids, and to $\Delta$-matroids. Below, in Proposition \ref{basic}, we have listed some well-known fundamental properties of $\Delta$-matroids.
In this note we show that $q$-$\Delta$-matroids can be defined such that they satisfy corresponding properties in the $q$-analogue. 

So far we have not found any concrete applications that call for the definition of $q$-$\Delta$-matroids, but we think that it is both fascinating and interesting to investigate how much of the material about $\Delta$-matroids can be given as a $q$-variant.

In Section \ref{three} the formal definition of $q$-$\Delta$ invariants is given in Definition \ref{def}, and mimics the definition of $\Delta$-matroids as defined above, and simultaneously a part of the axioms for $q$-matroids. We also discuss some examples of $q$-$\Delta$-matroids, as well as some problems with the definition of restriction.

In Section \ref{four} we  define $q$-$g$-matroids, which mimic $g$-matroids for sets, and in our case are defined by strong (identity) maps between ground spaces of pairs of $q$-matroids. A main result is that these are $q$-$\Delta$-matroids, as in the analogous situation for sets.
We also show that $q$-$g$-matroids satisfy analogous properties as in Proposition \ref{basic}.

We also introduce the concept of weak $q$-$g$-matroids, and study the inclusion relations that exist between the sets of  $q$-$g$-matroids, weak $q$-$g$-matroids and $q$-$\Delta$-matroids. These issues are direct analogues for spaces, of corresponding issues for sets.

There are some important concepts for $\Delta$-matroids, that we have not been able to mimic properly for $q$-$\Delta$-matroids. One of them is simply the symmetric difference $X \bigtriangleup Y$ for sets. Nowhere in our exposition do we use any $q$-variant of this. 

The second important concept follows in a natural way from this, that is the ``star''-operation $\bigtriangleup \rightarrow \bigtriangleup * A$, where each feasible set $F \in \F$ is replaced by $F \bigtriangleup A,$ for any fixed subset $A \subset E $ to give a new $\Delta$-matroid $\bigtriangleup * A$ on $E$.
This operation is sometimes called \emph{partial duality}, or \emph{twisting}.
The case $A=E$ gives usual duality: Each $F$ is replaced by $E-F$. In the $q$-case this is the only star-operation we have been able to mimic: We then let each feasible space $F$ be replaced by the orthogonal complement  $F^{\perp}$ with respect to the natural bilinear form on $E=\mathbb{F}^n.$ Thus we are able to define the dual of a $q$-$\Delta$-matroid, an important concept for us, that we use throughout.
Moreover we have not been able to  mimic the operations of restriction/deletion and contraction. The problems that arise when making the natural attempts to perform such operations, are explained in Subsection 3.3.

In Section \ref{five} we point out how pairs 
$C_1 \subset C_2$ of vector rank metric codes give rise to 
strong identity maps of $q$-matroids, and also to $q$-$g$-matroids, a perfect analogue to what pairs of error-correcting codes $C_1 \subset C_2$  with  Hamming distance  do. We also discuss possible definitions of representability for $q$-$\Delta$-matroids  and/or $q$-$g$-matroids, and a possible rank function. It is classical that a certain subclass of $\Delta$-matroids can by represented by skew-symmetric matrices. This form of representation goes via the star-operation, for which we have found no analogue for $q$-$\Delta$-matroids. See \cite{B1} and \cite[p. 47]{CMNR}. We suggest pairs of linear codes as a natural way of 
defining representability for $q$-$\Delta$-matroids.

It is unclear to what extent our $q$-analogues of $\Delta$-matroids and $g$-matroids interact with other kind of mathematics in the $q$-complex case. Are there, for example, meaningful definitions of $q$-graph and $q$-surface? Or can even $q$-$\Delta$-matroids be associated to usual graphs and surfaces in ways supplementing those of traditional $\Delta$-matroids? 
We leave this open for further research.

\section{Preliminaries} \label{2}

In this section we list the necessary preliminaries on $\Delta$-matroids, $q$-matroids, and strong maps.

\subsection{$\Delta$-matroids}

Throughout this subsection, we denote by $E$ a finite set, usually taken to be $\{1,\ldots,n\}$. Elements of this set are named by lower case letters like $x,y$ and $e$. We start with the definition of a $\Delta$-matroid.

\begin{definition} \label{firstdef}
A pair $(E,\mathcal{F})$, where $\mathcal{F}$ is a non-empty family of subsets of a finite \emph{ground set} $E$, is a $\Delta$-matroid if, for every two sets $X$ and $Y$ in $\mathcal{F}$, and for every element $x$ in their symmetric difference $ X \bigtriangleup Y= (X-Y) \cup (Y-X)$, there exists a $ y\in X\bigtriangleup Y$ such that $ X \bigtriangleup \{x,y\}$ is in the family. 
The non-empty family $\mathcal{F}$ is called the family, or set, of \emph{feasible sets} of the $\Delta$-matroid.
\end{definition}

We emphasise that the possibility $x=y$ is allowed in the definition of a $\Delta$-matroid. It is straightforward from this definition that the following holds.

\begin{proposition} \label{next}
The pair $\Delta^*=(E,\mathcal{F}^*)$ is  a $\Delta$-matroid,  if $\mathcal{F}^*$ is the family of set-theoretical complements of the sets that are members of $\mathcal{F}$, for a $\Delta$-matroid $\Delta=(E,\mathcal{F})$. The $\Delta$-matroid $\Delta^*=(E,\mathcal{F}^*)$ is called the dual $\Delta$-matroid  of $\Delta=(E,\mathcal{F})$, and vice versa.
\end{proposition}
\begin{proof}
For a more general statement, implying this duality, see for example \cite[Theorem 2.5]{CMNR}.
\end{proof}

The following result associates several $\Delta$-matroids to matroids and vice versa.

\begin{proposition} \label{basic}
\begin{enumerate}
\item If $\mathcal{B}$ is the family of bases for a matroid with ground set $E$, then $(E,\mathcal{B})$ is a $\Delta$-matroid.
\item If $\mathcal{I}$ is the family of independent sets for a matroid with ground set $E$, then $(E,\mathcal{I})$ is (also) a $\Delta$-matroid.
\item If $\mathcal{S}$ is the family of spanning sets for a matroid with ground set $E$, then $(E,\mathcal{S})$ is (also) a $\Delta$-matroid.
\item Let $\Delta=(E,\mathcal{F})$ be a $\Delta$-matroid.
Let $\mathcal{F}_U$ be the set of those elements of $\mathcal{F}$ that have the maximum cardinality among the elements of $\mathcal{F}$, and let $\mathcal{F}_L$ be the set of those elements of $\mathcal{F}$ that have the minimum cardinality among the elements of $\mathcal{F}$. Then $\Delta_U=(E,\mathcal{F}_U)$ and $\Delta_L=(E,\mathcal{F}_L)$ are both matroids (with $\mathcal{F}_U, \mathcal{F}_L$ as families of bases, respectively). These matroids are called the \emph{upper} and the \emph{lower} matroid of $\Delta$, respectively.
\end{enumerate}
\end{proposition}
\begin{proof}
For (1) we see that for an equicardinal family the defining condition for a $\Delta$-matroid is just the exchange property of matroids. For (2) we refer to \cite[Corollary 7.3]{B5}. Statement (3) follows from (2) and Proposition \ref{next}, since the spanning sets of a matroid are the complements of the independent sets of the dual matroid.

To prove (4) the equicardinality of the members of $\mathcal{F}_U$ and $\mathcal{F}_L$ is not enough, since the ``intermediate layers'' are not always basis sets for matroids. A finer study of Definition \ref{firstdef} does however reveal that we get just the exchange property for bases of a matroid in these two cases. (For the lower matroid: If $x \in X-Y$, then $y$ has to belong to $Y-X$. For the upper matroid: If $x \in Y-X$, then $y$ has to belong to $X-Y$.)
\end{proof}

We further define the following notions for $\Delta$-matroids, following \cite[Section 3]{IM}.

\begin{definition}
A \emph{loop} of a $\Delta$-matroid $\Delta=(E,\F)$ is an element $e \in E$, such that $e$ is contained in no feasible set of $\Delta$. A \emph{coloop} of a $\Delta$-matroid is a loop of its dual $\Delta$-matroid $(E,\F^*)$. An isthmus of a $\Delta$-matroid is an element $e \in E$ such that $e \in F$ for all $F \in \F$. 
\end{definition}

It follows directly from Proposition \ref{next} that coloops and isthmuses amount to the same; however, we will see that this is not the case in the $q$-analogue. The following two definitions give constructions to make a new $\Delta$-matroid from an existing one.

\begin{definition}\label{deltarestriction}
Let $e \in E$, for a $\Delta$-matroid $\Delta=(E,\F)$. The \emph{restriction} $\Delta|_{E-\{e\}}$ of $\Delta$ to $E-\{e\}$ is the $\Delta$-matroid $(E-\{e\},\F^'),$ where $\F^'$ is the set of all elements of $\F$ that do not contain $e$, if at least one such element exists. If no such element exists, then $e$ is a coloop of $\Delta$, and we let $\F^*$ consist of the set of all $F - \{e\}$ for all $ F \in \F.$
\end{definition}

In both cases the restriction of $\Delta$ to $E-\{e\}$ is also called the \emph{deletion} of $e$ from $\Delta$. 
We will avoid this terminology because it does not carry over well to the $q$-analogue.

\begin{definition}\label{deltacontraction}
Let $e \in E$, for a $\Delta$-matroid $\Delta=(E,\F)$. The \emph{contraction} $\Delta/\{e\}$ of $\{e\}$ form $\Delta$ is the $\Delta$-matroid $(E-\{e\},\F^')$, where $\F^'$ is the set of all elements of $\F$ that contain $e$ with $e$ removed from it, if at least one such element exist. If no such element exists, then $e$ is a loop of $\Delta$, and we let $\F'=\F$.
\end{definition}

It can be checked directly that restriction and contraction give indeed a $\Delta$-matroid. The operations of restriction and contraction are dual operations, that is:

\begin{proposition}
Let $\Delta$ be a $\Delta$-matroid and $e\in E$. Then $(\Delta/e)^*=\Delta^*|_{E-\{e\}}$.
\end{proposition}

\subsection{$q$-matroids}

The word \emph{$q$-matroid} refers to the $q$-analogue of matroids, that is, a generalization that replaces finite sets with \emph{finite dimensional vector spaces}. From now on, we will let $E$ denote the finite dimensional vector space $\mathbb{F}^n$ over some field $\mathbb{F}$. (From the context it should be clear when we are talking about a set $E$ or a space $E$.) We will denote subspaces of $E$ by upper case letters and $1$-dimensional subspaces of $E$ by lower case letters, mimicing the elements of a set. We fix a nondegenerate bilinear form $\perp$ on $E$.

$q$-Matroids have been first introduced by Crapo in \cite{crapo1964theory};  they gained importance due to the link to rank metric codes and network coding and we can find them reintroduced in \cite{JP18}.
\begin{remark}
The $q$-matroids described here, and defined in Definition \ref{rango} below, are not the same  as the objects called $q$-matroids by Bouchet in \cite[p.662]{B4}.   
\end{remark}
Let us start defining a $q$-matroid using the $q$-analogue of the rank function in classical matroid theory.

\begin{definition}\label{rango}
A $q$-matroid $M$ is a pair $(E,r)$ where $E$ is a finite dimensional vector space and $r$ is an integer-valued function defined on the subspaces of $E$ with the following properties:
\begin{itemize}
\item[(R1)] For every subspace $A$ of $E$, $0\leq r(A) \leq \dim A$. 
\item[(R2)] For all subspaces $A\subseteq B$  of $E$, $r(A)\leq r(B)$. 
\item[(R3)] For all subspaces $A,B$  of $E$, $r(A+ B)+r(A\cap B)\leq r(A)+r(B)$.  
\end{itemize}
The function $r$ is called the \emph{rank function} of the $q$-matroid. 
\end{definition}

As one can see, this is a straightforward $q$-analogue of the classical definition: simply size of a set becomes dimension of a vector subspace and the union of sets become the sum of vector subspaces.

Let us see -- similarly to the classical case of matroids -- some other objects that we can use to define a $q$-matroid.

\begin{definition}
Let $M=(E,r)$ be a $q$-matroid. A subspace $A \leq E$ is  an \emph{independent} space of $M$ if  its rank equals its dimension.
A maximal independent subspace  with respect to inclusion is called a \emph{basis} for $M$.
On the other hand, a subspace that is not an independent space of $M$ is  a \emph{dependent space} of $M$.\\
Given a subspace $C \leq E$, we call it a \emph{circuit} if it is dependent, but all proper subspaces of $C$ are independent.\\
We call a subspace $S$ a \emph{spanning space} of $M$   if  $r(S)=r(E)$.
A subspace $A$ of a $q$-matroid $(E,r)$ is called a \emph{flat} if for all $1$-dimensional subspaces $x$ of $E$ such that $x\nsubseteq A$ it holds $r(A+x)>r(A)$.
Finally, a maximal proper flat is called  $\emph{hyperplane}$.
\end{definition}
 All the objects listed above in this subsection  define a $q$-matroid via their own axioms systems, which turn out to be equivalent, their equivalence being called {\em cryptomorphism} (see \cite{bcj,CJ}).

Let us study some of these axiom systems, showing that they are not straightforward $q$-analogues of the classical cases.

\begin{definition}\label{indep-axioms}
	Let $\kI $ a family of subspaces of $E$. We define the following \emph{independence axioms}.
	\begin{itemize}
		\item[(I1)] $\I\neq\emptyset$.
		\item[(I2)] For all $I,J \subseteq E$, if $J\in\I$ and $I\subseteq J$, then $I\in\I$.
		\item[(nI3)] For all $I,J\in\kI$ satisfying $\dim I<\dim J$, there exists a codimension $1$ subspace $X\subseteq E$ with $I\subseteq X$, $J\not\subseteq X$ such that $I+x\in\kI$ for all $1$-dimensional $x\subseteq E$, $x\not\subseteq X$.
			\end{itemize}
	If $\kI$ satisfies the three  axioms above we say that $(E,\kI)$ is a collection of \emph{independent spaces}.
\end{definition}

\begin{definition}\label{indep-bases}
Let $\mathcal{B}$ be a family of subspaces of $E$.
We define the following \emph{basis axioms}.
\begin{itemize}
\item[(B1)] $\mathcal{B}\neq\emptyset$
\item[(B2)] For all $B_1,B_2\in\mathcal{B}$, if $B_1\subseteq B_2$, then $B_1=B_2$.
\item[(nB3)]  For all $B_1,B_2\in\kB$, and for each subspace $A$ that has codimension $1$ in $B_1$ there exists $X\subseteq E$ of codimension $1$ in $E$ such that $X \supseteq A$, $X\not \supseteq B_2$ and $A+x \in \mathcal{B}$ for all $1$-dimensional $x\subseteq E$, $x\not\subseteq X$.
\end{itemize}
If $\mathcal{B}$ satisfies the bases axioms above, we say that $(E,\mathcal{B})$ is a collection of \emph{bases}.
\end{definition}

The axiom (nB3), first stated in \cite{CJ}, is stated and used here in a different form. The present form corrects a mistake of \cite{CJ}, see also its Corrigendum \cite{CCJ}.

Apart from the case of the uniform matroids $U(0,n)$, whose only basis is the zero space  and $U(n,n)$, whose only basis is the ground space, a $q$-matroid cannot have only one basis.
Indeed, if there is only one basis, the independent spaces are exactly all its subspaces and in this case (nI3) is not satisfied.

The statement of \cite{CJ} would allow the presence of only one basis also in cases different from  $U(0,n)$ and  $U(n,n)$.
With the new statement, even if one takes $B_1=B_2$ many bases are constructed, and some of them are not the same as $B_1$.

\begin{definition}\label{spanning-axioms}
	Let $\kS$ be a family of subspaces of $E$. We define the following \emph{spanning space axioms}.
	\begin{itemize}
		\item[(S1)] $E \in \kS$.
		\item[(S2)] For all subspaces $I,J $ of $E$, if $J\in\kS$ and $J \subseteq I$, then  $I\in\kS$.
		\item[(nS3)] For all $S_1,S_2\in\mathcal{S}$ satisfying $\dim S_2<\dim S_1$, there exists a $1$-dimension subspace $x\subseteq S_1$, $x\not\subseteq S_2$ such that for all codimension-one $X\subseteq E$ with $X\not\supseteq x$ we have $X\cap S_1\in\kS$.
	\end{itemize}
	If $\kS$ satisfies the spanning axioms above, we say that $(E,\kS)$ is a collection of \emph{spanning spaces}.
\end{definition}

\begin{definition}\label{circuit-axioms}
Let $\mathcal{C}$ a family of subspaces of $E$. We
define the following \emph{circuit axioms}.
\begin{itemize}
\item[(C1)] $\{0\}\notin\mathcal{C}$.
\item[(C2)] For all $C_1,C_2\in\mathcal{C}$, if $C_1\subseteq C_2$, then $C_1=C_2$.
\item[(C3)]  For distinct $C_1,C_2 \in \mathcal{C}$ and any subspace $X$ of $E$ of codimension $1$ there is a circuit $C_3 \in \mathcal{C}$ such that $C_3 \subseteq (C_1+C_2)\cap X$.
\end{itemize}
If $\mathcal{C}$ satisfies the circuit axioms (C1)-(C3), we say that $(E,\mathcal{C})$ is a \emph{collection of circuits}.
\end{definition}

It is shown in \cite{bcj,CJ} that collections of independent spaces, bases, spanning spaces, and circuits, all define a $q$-matroid.

The following lemma is a straightforward $q$-analogue of a result for matroids. We include its proof here for completeness.

\begin{lemma} \label{fundcirc}
Let $M=(E,r)$ be a $q$-matroid and let $X$ be an independent space of $M$. Let $A\subseteq E$ such that $X\subseteq A$ is of codimension 1 in $A$ and $A$ is dependent. Then there is a unique circuit $C=C(X,A)$ contained in $A$. If $X$ is a basis, we call this circuit a \emph{fundamental circuit}.
\end{lemma}
\begin{proof}
Since $A$ is dependent, it contains a circuit $C_1$. This space is not contained in $X$, since $X$ is independent. Assume there is another circuit $C_2$ contained in $A$. Let $W$ be a codimension $1$ space of $E$, such that $X=A \cap W$. Then axiom (C3) in Definition \ref{circuit-axioms} gives that there is a new circuit space $C_3$ contained in $(C_1+C_2) \cap W=(C_1+C_2) \cap X \subset X$. That clearly is impossible, since $X$ is independent. Hence $C_1$ is the unique circuit contained in $A$.
\end{proof}

In the classical case, the set $A$ is usually written as $X\cup\{b\}$ for a certain $b\in E$. Then $C(X,A)$ is a circuit that contains $b$. In the $q$-analogue, we could write $A=X+b$ for some $1$-dimensional space $b$. This $b$ however is not unique, and it is \emph{not} guaranteed that $b\subseteq C(X,A)$. We therefore use the terminology as above.  

As with ``usual'' matroids, we can define the constructions of duality, restriction and contraction for $q$-matroids.

\begin{definition}\label{defdual}
Let $M=(E,r)$ be a $q$-matroid. Then $M^*=(E,r^*)$ is also a $q$-matroid, called the \emph{dual $q$-matroid}, with rank function
\[ r^*(A)=\dim(A)-r(E)+r(A^\perp). \]
\end{definition}

\begin{definition}\label{restr}
Let $M=(E,r)$ be a $q$-matroid. The \emph{restriction} of $M$ to a subspace $X$ is the $q$-matroid $M|_X$ with ground space $X$ and rank function $r_{M|_X}(A)=r_M(A)$, for all subspaces $A$ of $X$.
The \emph{contraction} of $M$ of a subspace $X$ is the $q$-matroid
$M/X$ with ground space $E/X$ and rank function $r_{M/X}(A/X)=r_M(A)-r_M(X)$ for every $A$ containing $X$.
\end{definition}

\begin{theorem}[Theorem 8 of \cite{CJ2}]
Restriction and contraction are dual operations, that is, $M^*/X=(M|_{X^{\perp}})^*$ and $(M/X)^*=M^*|_{X^{\perp}}.$
\end{theorem}

We finish this subsection with the following.

\begin{definition}\label{loopcoloop}
A $1$-dimensional subspace $\ell$ is called a \emph{loop} if $r(\ell)=0$. 
A \emph{coloop} of a $q$-matroid is a loop of the dual $q$-matroid.
\end{definition}

It is shown in \cite[Lemma 5.4]{JPR} that there cannot be a $1$-dimensional space in a $q$-matroid that is contained in every basis, unless the only basis is $E$. Therefore, there is no $q$-analogue of the notion of ``isthmus''. However, if a $q$-matroid $M$ has a coloop $e$, this means that $e^\perp$ is a codimension $1$ space in $E$ such that no basis of $M$ is contained in it.

\subsection{Strong maps}
In this last subsection we introduce the notion of strong and weak maps between $q$-matroids, following \cite{GJ}. We will then continue with proving some new results about strong maps that will be used in the sequel but that we feel are interesting in their own right as well.

For a $\mathbb{F}$-linear space $E$, we denote by $\kL(E)$ the lattice of its $\mathbb{F}$-linear subspaces.
\begin{definition} 
Let $\phi : E_1 \rightarrow E_2$ be a map between two finite-dimensional vector spaces over $\mathbb{F}_q$. We call $\phi$ an $\kL$-map if $\phi(V) \in \kL(E_2)$ for all $V \in \kL(E_1)$. The induced map $\kL(E_1)\rightarrow \kL(E_2)$ is denoted by $\phi_L$.
\end{definition}

If $E_1$ and $E_2$ are ground spaces of the $q$-matroids $M_1$ and $M_2$, an $\kL$-map $\phi : E_1 \rightarrow E_2$ can be viewed as a map $\phi : M_1 \rightarrow M_2$ between $q$-matroids.

\begin{definition} \label{strong2}
Let $M_i =(E_i,\rho_i)$ be $q$-matroids. Let $\phi : E_1\rightarrow E_2$ be
an $\kL$-map. We define the following types of maps.
 \begin{enumerate}
 \item $\phi$ is a strong map from $M_1$ to $M_2$ if $\phi^{-1}(F)$ is (a subspace of $E_1$ and) a flat for $M_1$ for all flats $F$ of $M_2$.
\item $\phi$ is a weak map $M_1 \rightarrow M_2$ if $\rho_2(\phi(V ))\le \rho_1(V)$ for all $V \in  \kL(E_1)$.
\end{enumerate}
\end{definition}
 
We prove the following result that characterizes strong maps. Th next three statements are $q$-analogues of \cite[Proposition 8.1.6]{kung1986}.
 
\begin{proposition} \label{equiv}
$Id: M_1\rightarrow M_2$ is a strong map if and only if 
\[ \rho_1(X)-\rho_1(Y) \ge \rho_2(X)-\rho_2(Y)\]
for all subspaces $Y\subseteq X\subseteq E$.
\end{proposition}
\begin{proof}
Assume $Id: M_1\rightarrow M_2$ is a strong map. In $M_1$, we can make a saturated chain of flats $\cl_{M_1}(Y)\subseteq F_0\subsetneq F_1\subsetneq F_2 \subsetneq \cdots \subsetneq F_k=\cl_{M_1}(X)$ where $\rho_1(Y)=\rho_1(F_0)$ and $\rho_1(F_k)=\rho_1(X)$. Then $\rho_1(X)-\rho_1(Y)=k$. Now we apply the strong map $Id$ to all flats in the chain and take all their closures in $M_2$. Since for any subspace $A\subseteq E$ we have that $\cl_{M_1}(A)\subseteq\cl_{M_2}(A)$, we find a new chain of flats $\cl_{M_2}(Y)=\cl_{M_2}(F_0)\subseteq \cl_{M_2}(F_1)\subseteq \cl_{M_2}(F_2) \subseteq \cdots \subseteq \cl_{M_2}(F_k)=\cl_{M_2}(X)$ that is saturated but might contain repeated elements. The length of this chain (after deleting repeated elements) is again the rank difference, so $\rho_2(X)-\rho_2(Y)\leq k$. This shows that $\rho_1(X)-\rho_1(Y) \ge \rho_2(X)-\rho_2(Y)$ for all subspaces $Y\subseteq X\subseteq E$.


Conversely we assume that $\rho_1(X)-\rho_1(Y) \ge \rho_2(X)-\rho_2(Y)$ whenever $Y \subseteq X$ for subspaces $X$, $Y$ of $E$.  Let $F$ be a ($q$-)flat for $M_2$. This means that $\rho_2(F+s)-\rho_2(F)=1$, for each one-dimensional subspace $s$ of $E$, not contained in $F$. But then $\rho_1(F+s)-\rho_1(F) \ge \rho_2(F+s)-\rho_2(F)=1$ also, so then $F$ is a flat for $M_1$ also.
\end{proof}

A direct consequence of the previous result is that the inverse of a strong map is a strong map between the duals.

\begin{lemma} \label{dualfact}
$Id: M_1 \rightarrow M_2$ is a strong map if and only if $Id: M_2^*\rightarrow M_1^*$ is a strong map.
\end{lemma}
\begin{proof}
The condition
\[\rho_2^*(Y)-\rho_2^*(X) \ge \rho_1^*(Y)-\rho_1^*(X),\]
for $X \subset Y$ is easily seen to be equivalent to 
\[\rho_1(X^{\perp})-\rho_1(Y^{\perp}) \ge \rho_2(X^{\perp})-\rho_2(Y^{\perp}),\]
where then $Y^{\perp} \subset X^{\perp}$.
\end{proof}

The following result is a corollary of the lemma above. It is a new characterisation of a strong map. Even though we will not make use of it in this paper, we feel it can be of interest in its own right.

\begin{lemma} \label{restriction}
$Id: M_1 \rightarrow M_2$ is a strong map if and only if every circuit of $M_1$ is the sum of circuits of $M_2$.
\end{lemma}

\begin{proof} By Lemma \ref{dualfact} it is equivalent to prove that $Id: M_2^*\rightarrow M_1^*$ is a strong map if and only if every circuit of $M_1$ is the sum of circuits of $M_2$. We see, by dualizing, that the latter is equivalent to saying that every hyperplane of $M_1^*$ is an intersection of hyperplanes of $M_2^*$, that is: a $q$-flat of $M_2^*$. 
This means if we start with a $(q$-)flat of $M_1^*$, this is an intersection of hyperplanes of $M_1^*$, and we have seen above that this implies that it is also an intersection of hyperplanes in $M_2^*$, which is a flat in $M_2^*$. So, every flat of $M_1^*$ is a flat of $M_2^*$, which is exactly the condition that $Id: M_2^* \rightarrow M_1^*$ is a strong map.
\end{proof}

The following result will be used frequently in our study of $q$-$g$-matroids in Section \ref{four}.
 
\begin{proposition} \label{argument}
Let $Id: E \rightarrow E$ be the identity map between (the ground set(s) of) two matroids $M_1=(E,\rho_1)$ and $M_2=(E,\rho_2). $ If $Id$ is a strong map, then every basis of $M_2$ is contained in a basis of $M_1$, and every basis of $M_1$ contains a basis of $M_2$.
\end{proposition}
 
\begin{proof}
By \cite[Page 9]{GJ} strong maps are weak. 
Let $B_2$ be a basis of $M_2$. Hence $\rho_1(B_2) \ge \rho_2(B_2)=\dim B_2$ for all bases $B_2$ of $M_2$. But this implies $\rho_1(B_2)=\dim B_2$ also, so $B_2$ is independent in $M_1$. Hence $B_2$ is contained in a basis of $M_1$. 
 
It remains to prove that every basis $B_1$ of $M_1$ contains a basis of $M_2$. This is the same as proving that every basis $B_1^*$ of $M_1^*$ is contained in a base $B_2^*$ of $M_2^*$. By Lemma \ref{dualfact}, $Id:M_2^*\rightarrow M_1^*$ is a strong map, so by the same reasoning as before, every basis $B_1^*$ of $M_1^*$ is contained in a base $B_2^*$ of $M_2^*$
\end{proof}

\section{$q$-$\Delta$-matroids} \label{three}
 
In this section we will define and study the $q$-analogue of $\Delta$-matroids.

\subsection{Defining the $q$-analogue of a $\Delta$-matroid.}
 
We propose the following definition for a $q$-analogue of $\Delta$-matroids.

\begin{definition} \label{def}
Let $E=\mathbb{F}^n$. The \emph{ground space} $E$, in combination with a non-empty family $\mathcal{F}$ of subspaces of $E$, is a \emph{$q$-$\Delta$-matroid} $(E,\mathcal{F})$ if the two following statements \textbf{(F1)} and \textbf{(F2)}, stated below, hold.
\begin{itemize}
\item[\textbf{(F1)}] For every two subspaces $X$ and $Y$ in $\mathcal{F}$, and for each subspace $A\subseteq E$ that has codimension $1$ in $X$, there either exists:
\begin{itemize}
\item[(i)] a codimension $1$ space $Z \subseteq E$ with $A \subseteq Z$ and $Y\not\subseteq Z$, such that for all 1-dimensional $z\subseteq E$, $z\not\subseteq Z$ it holds that $A+z\in \mathcal{F}$; \quad or
\item[(ii)] a codimension $1$ space $Z\subseteq E$ such that $Z \cap A \in \mathcal{F}$.
\end{itemize}
\item[\textbf{(F2)}] For every two subspaces $X$ and $Y$ in $\mathcal{F}$, and for each subspace $A\subseteq E$ with $X$ of codimension $1$ in $A$, there either exists:
\begin{itemize}
\item[(iii)] a $1$-dimensional $z\subseteq E$ with $z \subseteq A$, $z\not\subseteq Y$, such that for each $Z\subseteq E$ of codimension $1$, $z\not\subseteq Z$ it holds that $A \cap Z \in \mathcal{F}$; \quad or
\item[(iv)] a $1$-dimensional $z\subseteq E$ such that $A+z \in \mathcal{F}$.
\end{itemize}
\end{itemize}
\end{definition}

At first glance, this definition is a lot more cumbersome than the definition of a $\Delta$-matroid in Definition \ref{firstdef}. This is due to the lack of a suitable $q$-analogue of the symmetric difference. For an ``ordinary'' $\Delta$-matroid, we have that for every two feasible sets $X$ and $Y$, and for every element $x$ in their symmetric difference $ X \bigtriangleup Y$, there exists a $ y\in X\bigtriangleup Y$ such that $ X \bigtriangleup \{x,y\}$ is feasible. This property can be split in several cases, depending on if $x$ and $y$ are in $X-Y$ or in $Y-X$.

If $x\in X-Y$, the symmetric difference $X\bigtriangleup \{x,y\}$ is constructed as follows. First, it removes $x$ from $X$. This is reflected in \textbf{(F1)} above, forming $A$ of codimension $1$ in $X$. Then $y$ is considered. If $y\in Y-X$, $y$ is added to $X-\{x\}$, producing a feasible set of size $|X|$. The corresponding $q$-analogue is part (i) with $z\subseteq Y$. On the other hand, if $y\in X-Y$, $y$ is removed from $X-\{x\}$. This corresponds to part (ii), where the special case $x=y$ corresponds to $A\subseteq Z$.

If $x\in Y-X$, the symmetric difference $X\bigtriangleup \{x,y\}$ is constructed by first adding $x$ to $X$. This is reflected in \textbf{(F2)} above, forming $A$ that contains $X$ as a codimension $1$ space. Then $y$ is considered, and we have similar cases as before. If $y\in X-Y$, $y$ is removed from $X\cup\{x\}$, producing a feasible set of size $|X|$. The corresponding $q$-analogue is part (iii) with $X\not\subseteq Z$. On the other hand, if $y\in Y-X$, $y$ is added to $X\cup\{x\}$. This corresponds to part (iv), where the special case $x=y$ corresponds to $z\subseteq A$.

Note that for all cases, the dimension of the obtained feasible space coincides with the size of $X\bigtriangleup\{x,y\}$ in the classical case. A difference with the classical case is that (i) and (iii) produce a whole range of new feasible spaces, instead of just one. Also, in part (ii) and (iv), there is no dependence on $Y$. This is motivated by a similar difference between the classical basis axiom (B3) and the axiom (nB3) for $q$-matroids in Definition \ref{indep-bases}.

\begin{proposition} \label{qdual}
Given a $q$-$\Delta$-matroid $\Delta=(E,\mathcal{F})$. Let $\mathcal{F}^*=\{F^\perp:F\in\mathcal{F}\}$. Then $\Delta^*=(E,\mathcal{F}^*)$ is also a $q$-$\Delta$-matroid. We call $\Delta^*$ the \emph{dual $q$-$\Delta$-matroid} of $\Delta$.
\end{proposition}
\begin{proof}
In the definition of a $q$-$\Delta$-matroid, statements \textbf{(F1)} and \textbf{(F2)} are dual statements. That is, \textbf{(F1)} holds for $X,Y$ in $\mathcal{F}$ if and only if \textbf{(F2)} holds for $X^{\perp}, Y^{\perp}$ in $\mathcal{F}^*$.
\end{proof}

So, in the definition we could alternatively have written that \textbf{(F1)} holds both for $X,Y$ in $\mathcal{F}$ and for $X^{\perp}, Y^{\perp}$ in $\mathcal{F}^*$.

As mentioned in the introduction it is not clear (at least so far) that everything that works well for usual $\Delta$-matroids will work well for $q$-$\Delta$-matroids. The duality operation $\Delta \rightarrow \Delta^*$ is a special case of transforming a $\Delta$-matroid $\Delta$, to $\Delta*A$, where the feasible sets are the symmetric differences $A\bigtriangleup X$, for the $X$ appearing in $\mathcal{F}$. This idea of twisting was defined by Bouchet \cite{B3}.
The special case $A=E$ gives the set-theoretical complements $E-X$.  It is easy to find an analogue of this, in form of the orthogonal complements, $X^{\perp}$ as we indeed did in the case of $q$-$\Delta$-matroid, but less obvious how one could find an analogue of $\Delta*A$, for other $A$ than $E$.

\subsection{Examples from $q$-matroids}

Analogous to the classical case, we have the following results, where $q$-matroids directly give a $q$-$\Delta$-matroid, and the other way around. Propositions \ref{qmatroids}, \ref{ind}, \ref{span}, and \ref{extremes} give a complete $q$-analogue of Proposition \ref{basic}. 

\begin{proposition} \label{qmatroids}
If $(E,\mathcal{B})$ is a $q$-matroid, then it also is a $q$-$\Delta$-matroid.
\end{proposition}
\begin{proof}
Let $X,Y\in\mathcal{B}$ and let $A\subseteq X$ of codimension $1$. Then applying (nB3) gives that \textbf{(F1)} part (i) holds.
The statement of \textbf{(F2)} follows by viewing all pairs $X,Y$ as orthogonal complements of $X^{\perp}$ and  $Y^{\perp}$, and $X^{\perp}$ and $Y^{\perp}$ are two bases for the dual $q$-matroid.
\end{proof}

\begin{proposition} \label{ind}
If $\mathcal{I}$ is the family of independent spaces of a $q$-matroid, then $(E,\mathcal{I})$ is a $q$-$\Delta$-matroid.
\end{proposition}
\begin{proof}
For \textbf{(F1)} we will see that (ii) always holds.
Let $Z$ be any codimension 1 space in $E$. Then $Z \cap A$ is independent since $A \subseteq X$ is independent by axiom (I2) of Definition \ref{indep-axioms}.

For \textbf{(F2)} we will prove that (iii) always holds. Let $A$ be a space containing $X$ of codimension one. If $A$ is independent, then $A \cap Z$ is independent for every $Z$, so (iii) holds. So assume $A$ is dependent. By Lemma \ref{fundcirc} there is a unique circuit $C=C(X,A)$ in $A$ (obviously contained in neither $X$ nor $Y$). Let $z$ be contained in $C-Y$. If $Z$ is a codimension 1 space not containing $z$, then it intersects $A$ in a space that must be independent: if it was dependent, it would contain a circuit, and this a different circuit than $C$, since it does not contain $z$. This contradicts the property of a unique circuit $C(X,A)$. Hence (iii) holds.
\end{proof}

\begin{proposition} \label{span}
The pair $(E,\mathcal{S})$ is a $q$-$\Delta$-matroid, for $\mathcal{S}$ the set of spanning spaces for a $q$-matroid with ground space $E$.
\end{proposition}
\begin{proof}
It is well known that the family of the spanning spaces of a $q$-matroid is the set of orthogonal complements of the independent spaces of the dual $q$-matroid (see Definition \ref{defdual}). Hence the result follows from Propositions \ref{qdual}
and \ref{ind}.
\end{proof}

\begin{proposition} \label{extremes}
Let $\Delta=(E,\mathcal{F})$ be a $q$-$\Delta$-matroid.
Let $\mathcal{F}_U$ be the set of those elements of $\mathcal{F}$ that have the maximum dimension among the elements of $\mathcal{F}$, and let $\mathcal{F}_L$ be the set of those elements of $\mathcal{F}$ that have the minimum dimension among the elements of $\mathcal{F}$. Then $\Delta_U=(E,\mathcal{F}_U)$ and $\Delta_L=(E,\mathcal{F}_L)$ are both $q$-matroids (with $\mathcal{F}_U$ and $\mathcal{F}_L$ as families of bases, respectively). These matroids are called the \emph{upper} and the \emph{lower} $q$-matroid of $\Delta$, respectively.
\end{proposition}
\begin{proof}
Assume that both $X$ and $Y$ have minimal dimension. Then (ii) cannot hold. Since \textbf{(F1)} holds, we see that (i) holds. This is (nB3).
Since (B1) and (B2) obviously hold, these spaces of minimal dimension are the bases of a $q$-matroid.

Assume that that both $X$ and $Y$ have maximal dimension. Then (iv) cannot hold.
Since \textbf{(F2)} holds, we must have that (iii) holds. Statement (iii) is statement (i) for $X^{\perp}$ and $Y^{\perp}$ with respect to $\mathcal{F}^*$. Hence it is (nB3) for these spaces. Again (B1) and (B2) obviously hold for them. Hence the orthogonal complements of our spaces constitute the bases of a matroid. Therefore our original spaces do so too. 
\end{proof}

The above makes clear that (i) and (iii) are modeled after the basis exchange axiom (nB3), as in the classical case.

\subsection{$q$-$\Delta$-matroids and restriction: some examples}  

In this section we discuss examples of $q$-$\Delta$-matroids. We use them to show that a straightforward definition of contraction of a $q$-$\Delta$-matroid does not work.

\begin{proposition} \label{ex1}
Let $E=\mathbb{F}^4$ and let $\mathcal{F}$ be a family of subspaces of $E$ consisting of $\{0\},E$, and a family $\mathcal{D}$ of
$2$-dimensional spaces. We then have that $\Delta=(E,\mathcal{F})$ is a $q$-$\Delta$-matroid if and only if every $1$-dimensional subspace of $E$ is contained in some element of $\mathcal{D}$, and every $3$-dimensional subspace of $E$ contains an element of $\mathcal{D}$.
\end{proposition}

\begin{proof}
Let $X,Y \in \mathcal{F}$ ($X=Y$ is possible).
We have the following possibilities for $(\dim X, \dim Y)$: $(0,0),(0,2),(0,4),(2,0),(2,2),(2,4),(4,0),(4,2),(4,4)$.
 
For $(0,0),(0,2),(0,4)$  the condition \textbf{(F1)} is empty. For $(0,0)$ the condition \textbf{(F2)} always holds via (iii). For $(0,2)$ \textbf{(F2)}  holds via (iv) if (and only if) every one-dimensional subspace of $E$ is contained in some element of $\mathcal{D},$ and for  $(0,4)$ (iii) never holds, but \textbf{(F2)} always holds via (iv) if and only if and only if every one-dimensional subspace of $E$ is contained in some element of $\mathcal{D}.$  All in all \textbf{(F1)} and \textbf{(F2)} hold in all these three cases if and only if and only if every one-dimensional subspace of $E$ is contained in some element of $\mathcal{D}.$

For $(2,0),(2,2),(2,4)$ both \textbf{(F1)} and \textbf{(F2)} hold via (ii) and (iv), respectively.
 
For $(4,0),(4,2),(4,4)$ the condition \textbf{(F2)} is empty. For $(4,0)$ (i) never holds, but \textbf{(F1)} always holds via (ii) if and only if and only if every three-dimensional subspace of $E$ contains some element of $\mathcal{D}$. For $(4,2)$ (ii) holds if (and only)if and only if every three-dimensional subspace of $E$ contains some element of $\mathcal{D}$.

For $(4,4)$ the condition \textbf{(F1)} always holds via (i).
All in all \textbf{(F1)} and \textbf{(F2)} hold in all these three cases if and only if and only if every three-dimensional subspace of $E$ contains some element of $\mathcal{D}$
\end{proof}

Using the construction above, we give an example of a $q$-$\Delta$-matroid that does not come from a $q$-matroid.

\begin{example}\label{qdeltaex}
Let $E=\mathbb{F}^4$ and let $\mathcal{S}$ be a spread in $E$. That is, $\mathcal{S}$ is a collection of $2$-dimensional subspaces of $E$ such that every $1$-dimensional subspace of $E$ is contained in exactly one spread element. (To construct this, one could take for example the well-known geometric construction of a Desarguesian spread.) Since the orthogonal complements of spread elements in $\mathbb{F}^4$ form again a spread, we also have that every $3$-dimensional subspace of $E$ contains exactly one spread element.
Let $\mathcal{F}=\{\{0\},E\}\cup\mathcal{S}$. This is a $q$-$\delta$-matroid by the previous Proposition \ref{ex1}.
\end{example}

There are several remarks to be made about the previous example. First, it can be seen as a $q$-analogue of the $\Delta$-matroid with $E=\{1,2,3,4\}$ and $\mathcal{F}=\{\emptyset,\{1,2\},\{3,4\},\{1,2,3,4\}\}$, of which the property for feasible sets is easily checked. Secondly, we see that the the collection of $2$-dimensional spaces in $\mathcal{F}$, that is the spread $\mathcal{S}$, do not form the collection of bases of a $q$-matroid: axiom (nB3) does not hold for spread elements. The upper and lower $q$-matroid of the example are $U(4,4)$ and $U(0,4)$, respectively.

We wish to define restriction for a $q$-$\Delta$-matroid. Using the example above, we show that some seemingly straightforward definitions do not work.

\begin{proposition}
Let $(E,\mathcal{F})$ be a $q$-$\Delta$-matroid and let $T\subseteq E$ be a codimension $1$ space containing at least one element of $\mathcal{F}$. Define the family $\mathcal{F}'=\{F\in\mathcal{F}:F\subseteq T\}$. Then $(T,\mathcal{F}')$ is not necessarily a $q$-$\Delta$-matroid.
\end{proposition}
\begin{proof}
Consider the $q$-$\Delta$-matroid of Example \ref{qdeltaex}. Let $T\subseteq E$ be a subspace of codimension $1$, and so of dimension $3$. From the definition of $\mathcal{S}$,  it follows that there is exactly one element $S$ of $\mathcal{S}$ contained in $T$, so $\mathcal{F}'=\{\{0\},S\}$. We show that \textbf{(F2)} does not hold for $\mathcal{F}'$. Let $A$ be a subspace that contains $S$ of codimension $1$, that is, $A=T$. First we check (iii). For any $z\subseteq A$ and $Z\subseteq T$ of codimension $1$ not containing $z$, we need that $A\cap Z\in\mathcal{F}'$. This can not happen, since $A\cap Z$ has dimension $2$ and there is only one element of $\mathcal{F}'$ of dimension $2$, that is $S$. So (iv) needs to hold. But this can also not happen, since for any choice of $z$ we have that $A+z=A=T$, which is not in $\mathcal{F}'$. Hence $\mathcal{F}'$ is not the family of feasible spaces of a $q$-$\Delta$-matroid.
\end{proof}

Below we give a slightly less straightforward attempt for restriction, but this also does not give a $q$-$\Delta$-matroid.

\begin{proposition}
Let $(E,\mathcal{F})$ be a $q$-$\Delta$-matroid and let $T\subseteq E$ be a codimension $1$ space containing at least one element of $\mathcal{F}$. Define the family $\mathcal{F}'=\{F\cap T:F\in\mathcal{F}\}$. Then $(T,\mathcal{F}')$ is not necessarily a $q$-$\Delta$-matroid.
\end{proposition}
\begin{proof}
We take the same Example \ref{qdeltaex} as in the previous proposition. For any $F\in\mathcal{S}$ the intersection $F\cap T$ has dimension $1$ or $2$. Let $S\in\mathcal{S}$ be the unique spread element contained in $T$. Now we have that $\mathcal{F}'=\{\{0\},S, T\}\cup\{\langle x\rangle \subseteq T:x\not\in S\}$. Then for $S$ and $\{0\}$ property \textbf{(F2)} fails, as explained in the previous proof.
\end{proof}

These examples show that it is difficult to define restriction and contraction for $q$-$\Delta$-matroids in a way analogous to the corresponding definitions for $\Delta$-matroids, as given in Definitions \ref{deltarestriction} and \ref{deltacontraction}.

\section{A $q$-analogue of $g$-matroids and weak $g$-matroids} \label{four}

In the classical case, $\Delta$-matroids are related to several other objects, such as $g$-matroids and objects satisfying a variation of the axioms \textbf{(F1)} and \textbf{(F2)}. This section makes a $q$-analogue of these relations. An overview will be given in a diagram at the end of this section.

\subsection{$q$-$g$-matroids}

The concept of $g$-matroids is due to Tardos \cite{T}. It was later studied by Bouchet in \cite{B2,B3}. We give $q$-analogues of various definitions, leading to the definition of weak and strong $q$-$g$-matroids.

\begin{definition} \label{strong1}
Given a pair of $q$-matroids $M_1$ and $M_2$ with the same ground space $E$  such that any basis of $M_2$ is contained in a basis of $M_1$, and any basis of $M_1$ contains a basis of $M_2$.
The subspace system defined by such a pair of $q$-matroids is the set $\mathcal{F}$ of subspaces $F$ of $E$, such that there exists a basis $B_2$ of $M_2$, and a basis $B_1$ of $M_1$, such that $ B_2 \subseteq F \subseteq B_1$.
The pair $(E,\mathcal{F})$ is called a \emph{weak $q$-$g$-matroid}.
\end{definition}

We easily obtain the analogue of Propositions \ref{basic} and \ref{next} also for these objects.

\begin{proposition} \label{basic3}
\begin{enumerate}
\item If $\mathcal{B}$ is the family of bases for a $q$-matroid $M$ with ground set $E$, then $(E,\mathcal{B})$ is a weak $q$-$g$-matroid.
\item If $\mathcal{I}$ is the family of independent sets for a $q$-matroid $M_1$ with ground set $E$, then $(E,\mathcal{I})$ is (also) a weak $q$-$g$-matroid .
\item If $\mathcal{S}$ is the family of spanning sets for a $q$-matroid $M_2$ with ground set $E$, then $(E,\mathcal{S})$ is (also) a weak $q$-$g$-matroid .
\item Let $\Delta=(E,\mathcal{F})$ be a weak $q$-$g$-matroid, derived from matroids $M_1$ and $M_2$, as in Definition \ref{strong1}.
Let $\mathcal{F}_U$ be the set of those elements of $\mathcal{F}$ that have the maximum dimension among the elements of $\mathcal{F}$, and let $\mathcal{F}_L$ be the set of those elements of $\mathcal{F}$ that have the minimum dimension among the elements of $\mathcal{F}$.
Then $\Delta_U=(E,\mathcal{F}_U)$ and $\Delta_L=(E,\mathcal{F}_L)$ are both matroids (with $\mathcal{F}_U$ and $\mathcal{F}_L$ as families of bases, respectively). These $q$-matroids are called the upper, and the lower $q$-matroid of $\Delta$, respectively.
\item If $(E,\mathcal{F})$ is a weak $q$-$g$-matroid, then $(E,\mathcal{F}^\perp)$, where $\mathcal{F}^\perp$ consists of all orthogonal complements of members of $\mathcal{F}$, is a weak $q$-$g$-matroid as well.
\end{enumerate}
\end{proposition}

\begin{proof}
For (1) take $M=M_1=M_2$, for (2) take for $M_2$ the zero $q$-matroid $U(0,n)$, and let $M_1$ be the $q$-matroid  we are looking at; for (3) set $M_1=U(n,n)$, and let $M_2$ be the $q$-matroid we are looking at; and for (4) the upper $q$-matroid will be $M_1$, and the lower one will be $M_2$. Finally, for the duality result (5), we use that $B_1\subseteq F \subseteq B_2$ if and only if $B_2^\perp\subseteq F^\perp\subseteq B_1^\perp$ and that taking the orthogonal complement of a basis of a $q$-matroid gives as result a basis of the dual $q$-matroid.
\end{proof}

We now give an important definition.

\begin{definition} \label{qgdef}
Let $Id: M_1 \rightarrow M_2$ be a strong map between $q$-matroids $M_1$ and $M_2$ on the same ground space $E$. A \emph{$q$-$g$-matroid} is the space system where the feasible spaces are the independent ones for $M_1$, that are also spanning for $M_2$.
\end{definition}

We then immediately obtain:

\begin{proposition} \label{strongisweak}
A $q$-$g$-matroid is a weak $q$-$g$-matroid.
\end{proposition}
\begin{proof}
This follows from Proposition \ref{argument}.
\end{proof}

The last part of the next result follows essentially from Proposition \ref{strongisweak}:

\begin{proposition} \label{qgbeinggood}
Proposition \ref{basic3} holds when replacing ``weak $q$-$g$-matroid''
with ``$q$-$g$-matroid'' everywhere.
\end{proposition}

\begin{proof}
In order to show that a weak $q$-$g$-matroid is a $q$-$g$-matroid, we need to show that there is a strong map between the $q$-matroids $M_1$ and $M_2$ giving the weak $q$-$g$ matroid. By Proposition \ref{equiv}, this means we need for all $Y\subseteq X\subseteq E$ that
\[ \rho_1(X)-\rho_1(Y) \ge \rho_2(X)-\rho_2(Y). \]
In (1), taking $M_1=M_2$, one obviously has
that this statement holds, since the rank functions are equal, so $Id$ is a strong map. In (2) we have $M_2=U(0,0)$, so $r_2=0$, so the same inequality holds again. 
In (3) the feasible spaces of the weak $q$-$g$-matroid studied consist of the orthogonal complements of the independent spaces of the given matroid $M_2$. These complements form a $q$-$g$-matroid by (2), and since the set of orthogonal spaces of the feasible spaces of a  $q$-$g$-matroid, form a $q$-$g$-matroid, we are done. Statement (4) holds by definition: the upper $q$-matroid is $M_1$ and the lower $q$-matroid is $M_2$. Finally, (5) holds by Lemma \ref{dualfact}.
\end{proof}

The following result with proof is a $q$-version of a result/proof communicated to us by Steven D. Noble:

\begin{proposition} \label{weaknotstrong}
There exists a weak $q$-$g$-matroid that is not a $q$-$g$-matroid.
\end{proposition}
\begin{proof}

Let $E=\Fq^4$, and let $M_1=U(2,3)\oplus U(1,1)$, and $M_2=U(1,2)\oplus U(1,2)$.
These $q$-matroids are defined in the sense of \cite[Section 6]{CJ2}. Let $e_1,e_2,e_3,e_4$ be coordinate vectors for $E$, where $e_1,e_2$ and $e_3,e_4$ are coordinates for the ground spaces of the two copies of $U(1,2)$, and $e_1,e_2,e_3$ for the ground space of $U(2,3)$ and $e_4$ corresponds to $U(1,1)$.

From \cite[Example 49]{CJ2} it follows that the bases of $M_2$ are all $2$-dimensional subspaces of $E$, except $E_1=\langle e_1,e_2\rangle$ and $E_2=\langle e_3,e_4\rangle$.
In virtue of being the ground spaces of the two copies of $U(1,2)$ in the first place, their ranks are $r_2(E_1)=r_2(E_2)=1$.

After a similar calculation as that in \cite[Example 49]{CJ2} one arrives at the conclusion that the bases of  $M_1=U(2,3)\oplus U(1,1)$ are all $3$-dimensional subspaces of $E$, except $E^'_1=\langle e_1,e_2,e_3\rangle$, and that $r_1(E^'_1)=2$.

It is clear that all $3$-dimensional subspaces of $E$ contain $2$-dimensional subspaces different from $E_1$ and $E_2$,
so every basis of $M_1$ contains a basis of $M_2$.  
Moreover it is clear that each $2$-dimensional subspace of $E$, in particular any basis of $M_2$, is contained in more than one $3$-dimensional subspace of $E$, so in particular one different from $E^'_1$, i.e. in a basis of $M_1$.
Thus the space system of all spaces contained in a basis of $M_1$ and containing a basis of $M_2$, is a weak $q$-$g$-matroid.

However, there is no strong map between $M_1$ and $M_2$. Apply the criterion from Proposition \ref{equiv} to $X=E_1'$ and $Y=E_1$. Then we get
\[ r_1(E_1')-r_1(E_1)=2-2=0< 1=2-1= r_2(E_1')-r_2(E_1), \]
contradicting the criterion.
Thus the space system of all spaces contained in a basis of $M_1$, and containing a basis of $M_2$ is not a $q$-$g$-matroid.
\end{proof}

\subsection{Relation to deltamatroids}
Working with sets, as in \cite{B2} and \cite{B3}, instead of spaces as we do,  it has been proven in (\cite[p.70]{B2}) that not all $\Delta$-matroids are $g$-matroids. 
  
On the other hand it has been shown (see \cite[Proposition 7.2]{B3}) that $g$-matroids always are $\Delta$-matroids. It is therefore natural to give the following result (adapted to our setting, after an example given to us by Steven D. Noble):

\begin{proposition}
There exists a $q$-$\Delta$-matroid that is not an (even weak) $q$-$g$-matroid.
\end{proposition}
\begin{proof}
Let $E=\mathbb{F}^4$ and let $\mathcal{F}$ be the family containing all subspaces of even dimension of $E$. Then $(E,\mathcal{F})$ is a $q$-$\Delta$ matroid by Proposition \ref{ex1}. On the other hand, it is clearly not a (weak) $q$-$g$-matroid, because $\mathcal{F}$ contains $\{0\}$ and $E$, but does not contain any subspaces of odd dimension.
\end{proof}

Furthermore we have:

\begin{proposition} \label{weakgnotdelta}
There exists a weak  $q$-$g$-matroid, which is not a $q$-$\Delta$-matroid.
\end{proposition}

\begin{proof}
Take the weak $q$-$g$-matroid from Proposition \ref{weaknotstrong}. Its feasible spaces are all bases of $M_1$, that have dimension $3$, and all bases of $M_2$, that have dimension $2$.
Let $X=\langle e_1,e_2,e_4 \rangle$, and $Y=\langle e_1,e_3\rangle$.
We see that $X$ is a basis for $M_1$, and $Y$ is a basis for $M_2$, so both are feasible.
Look at the codimension $1$ space $A=\langle e_1,e_2\rangle$ in $X$. Certainly $A$ is not feasible, since we already saw $A=E_1$ is not a basis of $M_1$ or $M_2$. Any intersection of $A$ by a codimension $1$-plane $Z$ in $E$ is also then not feasible.
Hence (ii) of \textbf{(F1)} does not hold. So (i) of \textbf{(F1)} needs to hold. In particular, there is a $1$-dimensional space $z \subseteq Y$ such that $A+z$ is feasible. But then $A+z=\langle e_1,e_2,e_3 \rangle$ since both $A$ and $z$ are contained in this space, and $A$ is not feasible. But $\langle e_1,e_2,e_3 \rangle$ cannot be feasible. It has dimension $3$ and cannot be a basis of $M_2$, which has rank $2$. And it cannot have rank $3$ for $M_1$, since it is completely contained in one of the summands (which in this case has rank $2$ for $M_1$).
\end{proof}

In \cite[p. 157]{B3} (after \cite{T}) one gives an axiomatic description of a $g$-matroid, which as a starting point is defined in an analogous way for sets   as we have done for spaces, in Definition \ref{qgdef}. Here is the alternative,  axiomatic definition of a $g$-matroid $(E,\mathcal{F})$ for a finite set $E$ from \cite{B3}.

\begin{definition} \label{axsets}
For all $X,Y$ in $\mathcal{F}$, we have:

 (3) If $x \in X-Y$, then either $X-x \in \mathcal{F}$ or there exists a $y \in Y-X$ such that $X-x+y \in \mathcal{F}$.
 
 (4) If $x \in X-Y$, then either $Y+x  \in \mathcal{F},$ or there exists a $y \in Y-X$ such that $Y+x-y \in \mathcal{F}$.

\end{definition}

If (3) and (4) hold, then one easily sees that $(E,\mathcal{F})$ is a $\Delta$-matroid in the traditional case of sets. In the $q$-analogue it is natural to define analogues of axioms (3) and (4):

\begin{definition}
Let $E=\mathbb{F}^n$. 
We define the following two properties for a family $\mathcal{F}$ of subspaces of $E$:

\begin{itemize}
\item[\textbf{(F3)}] For every two subspaces $X$ and $Y$ in $\mathcal{F}$, and for each subspace $A\subseteq E$ that has codimension $1$ in $X$, it either holds that:
\begin{itemize}
\item[(i)] there exists a codimension $1$ space $Z \subseteq E$ with $A \subseteq Z$ and $Y\not\subseteq Z$, such that for all $z\subseteq E$, $z\not\subseteq Z$ it holds that $A+z\in \mathcal{F}$; \quad or
\item[(v)] $A \in \mathcal{F}$.
\end{itemize}
\item[\textbf{(F4)}] For every two subspaces $X$ and $Y$ in $\mathcal{F}$, and for each subspace $A\subseteq E$ with $X$ of codimension $1$ in $A$, it either holds that:
\begin{itemize}
\item[(iii)] a $1$-dimensional $z\subseteq E$ with $z \subseteq A$, $z\not\subseteq Y$, such that for each $Z\subseteq E$ of codimension $1$, $z\not\subseteq Z$ it holds that $A \cap Z \in \mathcal{F}$; \quad or
\item[(vi)] $A\in \mathcal{F}$.
\end{itemize}
\end{itemize}
\end{definition}

\begin{remark} \label{F3F4Delta}
Any space system satisfying \textbf{(F3)} and \textbf{(F4)} will always be the set of feasible spaces for a $q$-$\Delta$-matroid, since (v) implies (ii) and (vi) implies (iv). 
\end{remark}

\begin{theorem} \label{main?}
A $q$-$g$-matroid $G=(E,\mathcal{F})$ is a $q$-$\Delta$-matroid
\end{theorem}

\begin{proof}
We will prove that $\mathcal{F}$ satisfies ($\mathbf{F3}$) and ($\mathbf{F4}$), which we have seen imply ($\mathbf{F1}$) and ($\mathbf{F2}$). We will utilize that $G$ is a weak $q$-$g$-matroid associated to a pair $M_1,M_2$ of $q$-matroids, and with the extra condition that 
\[ \rho_1(X)-\rho_1(Y) \ge \rho_2(X)-\rho_2(Y) \]
whenever $Y \subseteq X$ for all pairs of  subspaces $X$, $Y$ of $E$.

First we show that $\mathcal{F}$ satisfies
\textbf{(F3)}. Let $X,Y$ be a pair of subspaces of $E$, and $A$ a subspace of codimension $1$ in $X$. If $A$ contains a basis of $M_2$, then we are done, since then $A$ also is contained in a basis of $M_1$, since $X$ is. Hence (v) of \textbf{(F3)} holds. So now assume that $A$ does not contain any basis of $M_2$. But $X$ does contain a basis $B_1$ of $M_2$, that is thus not contained in $A$. Likewise $Y$ contains a basis $B_2$ of $M_2$. Let $A_1 =A \cap B_1$. Then $A_1$ is of codimension $1$ in $B_1$, since $B_1$ is contained in $X$ but not in $A$, and $A$ is of codimension $1$ in $X$. Then by axiom (nB3) for $q$-matroids there exists a $Z \subseteq E$ of codimension $1$ in $E$ such that $A_1 \subseteq Z$, and
$Z$ does not contain $B_2$, and $A_1+z$ is a basis of $M_2$ for all $1$-dimensional $z\subseteq E$, and $z$ not contained in $Z$. We observe that $Z$ does not contain $Y$, since it does not contain $B_2$. 

We argue that $Z$ contains $A$. Suppose not, then there would be a $z$ outside $Z$ but in $A$ such that $A_1+z$ is a basis for $M_2$, as we have seen. But $A_1+z\subseteq A$, which contradicts our assumption that $A$ does not contain a basis of $M_2$.

So $Z$ contains $A$, and we may argue like this to prove (i) of \textbf{(F3)}: It holds if each $A+z$ not only contains a basis of $M_2$, as we have proved (since each $A_1+z$ contains such a basis), but also is contained in a basis of $M_1$. This happens if and only if $A+z$ is independent in $M_1$. 

But $\rho_2(A + z)-\rho_2(A) \ge \rho_2(M_2)-(\rho_2(M_2)-1)=1$, since $A+z$ contains $A_1+z$, which contains a basis of $M_2$, while $A$ contains no basis of $M_2$. Hence by the rank inequality $\rho_1(A+z)-\rho_1(A) \ge 1$ also. But $\rho_1(A)=\dim(A)=\dim(X)-1$, since $X$ is independent in $M_1$. Hence $\rho_1(A +z) \ge \rho_1(A)+1=\dim (X)=\dim(A+z)$. Hence $A+z$ is independent in $M_1$, and therefore contained in a basis of $M_1$, for all the relevant $z$, and so they are feasible for $\mathcal{F}$.

By Lemma \ref{dualfact} the dual $q$-$g$-matroid $G^{\perp}=(E,\mathcal{F}^{\perp})$ is a $q$-$g$-matroid too, corresponding to a strong map $M_2^{\perp} \rightarrow M_1^{\perp}$. Using this, we see that \textbf{(F4)} is just \textbf{(F3)} for $X^{\perp}, Y^{\perp}$, so 
\textbf{(F4)} holds for $G$, since \textbf{(F3)} holds for $G^{\perp}$.
\end{proof}

Theorem \ref{main?} gives new proofs of Propositions \ref{ind} and \ref{span}, since the space systems formed by independent spaces, and the space systems formed by spanning spaces of a $q$-matroid, are $q$-g-matroids, by Proposition \ref{qgbeinggood}, and
$q$-$g$-matroids are $q$-$\Delta$-matroids, by Theorem \ref{main?}.

We may now look at 4 kinds of objects: $q$-$\Delta$-matroids, $q$-$g$-matroids, weak $q$-$g$-matroids, and space systems satisfying ($\mathbf{F3}$) and ($\mathbf{F4}$). We know:
\begin{itemize}
\item $q$-$g$-matroids are both $q$-$\Delta$-matroids, weak $q$-$g$-matroids, and space systems satisfying ($\mathbf{F3}$) and ($\mathbf{F4}$).
\item Space systems satisfying ($\mathbf{F3}$) and ($\mathbf{F4}$) are $q$-$\Delta$-matroids.
\item $q$-$\Delta$-matroids are not always weak $q$-$g$-matroids, and therefore not always $q$-$g$-matroids.
\item Weak $q$-$g$-matroids are not always $q$-$\Delta$-matroids, and weak $q$-$g$-matroids are not always $q$-$g$-matroids.
\end{itemize}
We believe that: Space systems satisfying ($\mathbf{F3}$) and ($\mathbf{F4}$) are always 
(and are then the same as) $q$-$g$-matroids, but this is still an open question (to us).

\subsection{Partial results about equivalences}

We now take a closer look at the (in)equivalence between $q$-$\Delta$-matroids, $q$-$g$-matroids, weak $q$-$g$-matroids, and space systems satisfying ($\mathbf{F3}$) and ($\mathbf{F4}$).

Given an arbitrary $q$-$\Delta$-matroid, with an upper $q$-matroid $M_1$ and a lower $q$-matroid $M_2$. Two questions one might pose, are the following:
\begin{itemize}
\item[(i)] Is the map $Id: M_1 \rightarrow M_2$ a strong map? Or at least:
\item[(ii)] Is every basis of $M_2$ contained in a basis of $M_1$, and does every basis of $M_1$ contain a basis of $M_2$?
\end{itemize}

Unfortunately, we have not been able to prove question (i) above, but we can prove a positive answer to question (ii).

\begin{proposition} \label{weakversion}
Let $(E,\mathcal{F})$ be a $q$-$\Delta$-matroid. Then every feasible space, in particular each basis of its lower $q$-matroid $M_2$, is contained in a basis of its upper $q$-matroid $M_1$, and every feasible space, in particular each basis of $M_1$, contains a basis of $M_2$.
\end{proposition}

\begin{proof}
We want to prove that every feasible space $Y$ is contained in a feasible space $W$ of maximal dimension. Let $X=X_0$ be any feasible space of maximal dimension. If $X$ contains $Y$, done, put $W=X$. This happens in particular if $X=E$, so we may assume that $X$ is strictly contained in $E$.
If $X$ does not contain $Y$, there is a $1$-dimensional space $y$ in $Y-X$. Define $s= \dim Y- \dim (Y \cap X) >0$. 
Let $A=X+y$, so $X$ is of codimension $1$ in $A$.
Then axiom \textbf{(F2)} comes into play, and says that (iii) or (iv) occurs. But (iv) says that $A+z$ is feasible for some $z$. This is impossible, even in the case that $z \subseteq A$, since $\dim A=\dim X + 1$, and $X$ is maximal dimensional among feasible sets.

Hence (iii) holds: there is a $1$-dimensional space $z$ with $z\subseteq A$, $z\not\subseteq Y$ such that for each $Z\subseteq E$ of condimension $1$, $z\not\subseteq Z$, it holds that $A\cap Z$ is feasible. In particular, there is such a $Z$ that contains $Y$, since $z$ is not contained in $Y$. 
Set $X_1=Z \cap A$. We see that $\dim Y - \dim (X_1 \cap Y) \le s-1$, since obviously $\dim Y - \dim (A \cap Y) \le s-1$, and $A \cap Y = A \cap Z \cap Y=X_1 \cap Y$, so $\dim Y - \dim (Y \cap X_1) \le s-1$.   
(We have ``eaten'' $y \in Y-X$.) After $s$ steps like this one, we end up with a feasible space $W=X_s$ containing $Y$.

Since the orthogonal complements of feasible spaces $Y$ of a $q$-$\Delta$-matroid constitute the feasible sets of a new $q$-$\Delta$-matroid, all complements of such feasible set are contained in a basis of the upper matroid of this new $q$-$\Delta$, by the conclusion above. But these new upper bases are complements of the bases of the lower matroid of the original $q$-$\Delta$-matroid, so the $Y$ we started with contains such a basis of the lower matroid of the original $q$-$\Delta$-matroid also, in other words $Y$ contains a feasible set of minimal cardinality.
\end{proof}

\begin{corollary} \label{consolation}
The set of feasible spaces of a  $q$-$\Delta$-matroid is a subset of the set feasible subspaces of a weak $q$-g-matroid, namely the one formed by its own upper and lower $q$-matroids.
\end{corollary}

 It is important to note that Proposition \ref{weakversion} (also in view of Proposition \ref{weakgnotdelta} and Corollary \ref{consolation}) does not imply  that every $q$-$\Delta$-matroid is a weak $q$-g-matroid, since such a weak $q$-g-matroid has $\it{all}$ spaces contained in a basis space of $M_1$ and containing a basis space of $M_2$ as its set of feasible spaces, and that is not true for $q$-$\Delta$-matroids in general. This observation paves the way for the following definition.
 
\begin{definition} \label{saturated}
A $q$-$\Delta$-matroid is \emph{saturated} if for every triple $Y \subseteq Z \subseteq X$ of spaces, where $X$ and $Y$ are feasible, we have $Z$ feasible also.
\end{definition}

We have:
\begin{proposition} \label{something}
A $q$-$\Delta$-matroid is saturated if it satisfies axioms \textbf{(F3)} or \textbf{(F4)}.
\end{proposition}

\begin{proof}
Assume \textbf{(F3)} holds. Let $X$ and $Y$ be two feasible spaces, with $Y \subseteq X$.  
We will prove that every codimension $1$ space $A$ in $X$, containing $Y$, is feasible. If so, by iterating the process, we can show that every space $Z$, with $ Y \subseteq Z \subseteq X$ is feasible, and hence the $q$-$\Delta$-matroid is saturated.
We see that $A$ satisfies the assumption of \textbf{(F3)}: $A$ has codimension $1$ in $X$.
Then by (v) of \textbf{(F3)} we have: $A$ is feasible, and we are done, or: (i) holds. But (i) implies that there exist a $Z \subseteq E$ of codimension $1$ in $E$ such that $A \subseteq Z$, and $Z$ does not contain $Y$. But this is impossible, since $Y \subseteq A$. Hence (i) cannot hold, and (v) holds, and $A$ is feasible. 

Arguing in a dual way, we see that \textbf{(F4)} also implies that the $q$-$\Delta$-matroid is saturated.
\end{proof}

If we had been able to prove a stronger version of Proposition \ref{weakversion}, namely that for any $q$-$\Delta$-matroid, the map: $M_1 \rightarrow M_2$ is a strong map (the $q$-analogue of \cite[Theorem 3.3]{B5}), then we would have been able to prove that the converse of Proposition \ref{something} holds. 

To be precise: Under the hypothetical condition that  for any $q$-$\Delta$-matroid, the map: $M_1 \rightarrow M_2$ is a strong map, we have:  If a $q$-$\Delta$-matroid is saturated, then it satisfies \textbf{(F3)}  or \textbf{(F4)}, in fact both, since it is then a $q$-g-matroid (using the proof of Proposition \ref{main?} once more). In the classical case it has been proven that for a $\Delta$-matroid  $\Delta$ the map $Id$ from the upper to the lower matroid of $\Delta$ indeed is a strong map.

We finish this section with Figure \ref{diagram2}, that shows an overview of the relations between different objects as discussed in this section.

	\begin{figure}[ht]
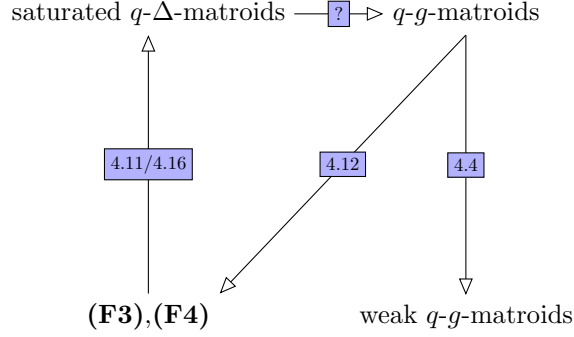

		\tikz[every node/.style={}, xscale=2.1]
		{
		 	\node (qd) at (2,2)  { saturated $q$-$\Delta$-matroids};
			 \node (f3f4) at (2,-2)  { \textbf{(F3)},\textbf{(F4)} };
		 	\node (qg) at (4,2)  { $q$-$g$-matroids  };
			 \node (wqg) at (4,-2)  { weak $q$-$g$-matroids };
			\draw [ -open triangle 45](qd) -- node[auto, rectangle, draw,fill=blue!30, scale=0.7,anchor=center] {?}  (qg);
   
\draw [-open triangle 45](f3f4) --  node[auto, rectangle, draw,fill=blue!30, scale=0.7,anchor=center] {\ref{F3F4Delta}/\ref{something}}  (qd);

\draw [-open triangle 45](qg.south) -- node[auto, rectangle, draw,fill=blue!30, scale=0.7,anchor=center] {\ref{strongisweak}}   (wqg.north);
\draw [-open triangle 45](qg.south) --node[auto, rectangle, draw,fill=blue!30, scale=0.7,anchor=center] {\ref{main?}}    (f3f4.north east);
		}
		\caption{Relations between various structures. An arrow reads ``is a''. The `?' indicates that we do not know if this relation is true.}
		\label{diagram2} 
	\end{figure}

\section{Representability and rank} \label{five}

As potential topics for further investigation, we describe how one might consider representability and rank for a $q$-$\Delta$-matroid.

\subsection{Pairs of Codes}

Let $C_2 \subset C_1$ be an inclusion of two 
linear codes with the Hamming metric over a field $\mathbb{F}_q$, and let $\mathcal{C}_2 \subset \mathcal{C}_1$ be an inclusion of two $\mathbb{F}_{q^m}$-linear rank-metric codes over $\mathbb{F}_q$. Let $r_2,r_1$ be the rank functions of the associated matroids $M_2$ and $M_1$ of $C_2,C_1$, respectively, and let $\rho_2,\rho_1$ be the rank functions of the associated $q$-matroids $\mathcal{M}_2$ and $\mathcal{M}_1$ of  $\mathcal{C}_2,\mathcal{C}_1,$ respectively. The following is well known \cite{oxley2006matroid,JP18}.

\begin{proposition}
Let $C$ be a code in the Hamming metric with dual $C^\perp$ (with respect to the standard   inner product in $\mathbb{F}_q$) and let $\mathcal{C}$ be a $\mathbb{F}_{q^m}$-linear code in the rank metric with dual $\mathcal{C}^\perp$ (with respect to some non-degenerate bilinear form on $\mathbb{F}_{q^m}$). Then we have:
\begin{enumerate}
\item The dual matroid $M_C^*$ of $M_C$ is the matroid of $M_{C^{\perp}}$.
\item The dual $q$-matroid $\mathcal{M}_{\mathcal{C}}^*$ of $\mathcal{M}_{\mathcal{C}}$ is the $q$-matroid of $\mathcal{M}_{\mathcal{C}^{\perp}}$ for a $\mathbb{F}_{q^m}$-linear rank-metric code $\mathcal{C}$.
\end{enumerate}
\end{proposition}

We can use this result to prove the following about nested codes.

\begin{proposition} \label{pairs}
\begin{enumerate}
\item \label{hamming} If $M_1,M_2$ are matroids of  generator matrices of linear error-correcting codes $C_1,C_2$ for the Hamming distance, and $C_2$ is an $F_q$-subspace of $C_1$, then $Id: M_1\rightarrow M_2$  is a strong map of matroids.
\item \label{rankmetric} If $\M_1,\M_2$ are matroids of generator matrices of Gabidulin rank-metric codes $C_1,C_2$, and $C_2$ is an $F_{q^m}$-subspace of $C_1$, then   $Id: \mathcal{M}_1\rightarrow \mathcal{M}_2$ is a strong map of $q$-matroids.
\end{enumerate}
\end{proposition}

\begin{proof}
We first prove part (\ref{rankmetric}): We will use Proposition \ref{equiv} and show that $\rho_1(X)-\rho_1(Y) \ge \rho_2(X)-\rho_2(Y)$, whenever $Y \subseteq X$ for subspaces $X$, $Y$ of $E$. Calculations of dimensions will be over $\mathbb{F}_{q^m}$. We denote by $C(J)$ all codewords in a code $C$ with support contained in $J$.

We have $\rho_1(X) -\rho_1(Y)= (\dim C_1 - \dim C_1(X^{\perp})) -(\dim C_1 - \dim C_1(Y^{\perp}))= \dim C_1(Y^{\perp})-\dim C_1(X^{\perp}).$ Likewise $\rho_2(X) -\rho_2(Y)=\dim C_2(Y^{\perp})-\dim C_2(X^{\perp})$.

Clearly it is enough to show: $\dim A \le \dim B$, where $A = C_2(Y^{\perp})/C_2(X^{\perp})$ and $B=  C_1(Y^{\perp})/C_1(X^{\perp})$. This we do by describing an injective, linear map from $A$ to $B$.
First send each element of $C_2(Y^{\perp})$ to $C_1(Y^{\perp})$ and then send the image to its class modulo $C_1(Y^{\perp})$ in $B$. This gives a map $C_2(Y^{\perp}) \rightarrow B.$ It is clear that if $v-w \in C_2(X^{\perp})$, then $v-w \in C_1(X^{\perp})$ also, so $v$ and $w$ are mapped to the same element of $B$. Hence this map can be viewed as a linear map $\phi:A \rightarrow B.$ If $v$ and $w$ are in different cosets relative to $C_2(X^{\perp})$, then $v-w$ is not contained in $C_2(X^{\perp})$, and then it is also not contained in $C_1(X^{\perp})$, since if $v-w \in C_1(X^{\perp})$, then $v-w \in C_2 \cap C_1(X^{\perp})= C_2(X^{\perp})$. Hence $\phi$ is injective, and $\dim A \le \dim B$.

Part ($\ref{hamming}$) is proved in an analogous way, recalling that $r_i(X)= \dim C -\dim C(E-X)$ for subsets $X \subset E$, for $i=1,2$. Calculations of dimensions will be over $\mathbb{F}_{q}$ here.
\end{proof}

By Proposition \ref{pairs} we see that pairs of codes as above define set/space systems associated to pairs of ($q$-)matroids, with $Id$ a strong map. Thus they do  not only determine ($q$-)demi-matroids, as in \cite[Page 987]{BJM} (and its $q$-counterpart), but also ($q$-)$\Delta$-matroids, that are even ($q$-)$g$-matroids. It is natural for a given $\Delta$-matroid to think of the existence of such a pair as in part ($\ref{hamming}$) giving rise to it, as representability of the $\Delta$-matroid over $\mathbb{F}_q$, and then as a $g$-matroid in the sense of \cite[Page 157]{B3}. Likewise one can think of an analogous pair of $\mathbb{F}_{q^m}$-linear rank-metric codes as representability of a $q$-$\Delta$-matroid over the field extension $\mathbb{F}_{q^m}/\mathbb{F}_q$.

In fact one could define two different notions of representability of $q$-$\Delta$-matroids in terms of pairs of codes. The weak version is that the upper $q$-matroid comes from an $\mathbb{F}_{q^m}$-linear rank-metric code $C_1$ and that the lower $q$-matroid comes from an $\mathbb{F}_{q^m}$-linear rank-metric code $C_2,$ where $C_2 \subset C_1$.

The strong version is that, in this situation, the $q$-$\Delta$-matroid in question is the entire $q$-$g$-matroid, whose feasible sets are the ones that are spanning for the lower matroid and independent for the upper one.
 
It should be noted that in the ``classical'' literature about $\Delta$-matroids, representability is not defined in terms of pairs of linear codes with Hamming distance. It is defined in terms of skew-symmetric matrices. See \cite{B1}, \cite[Subsection 5.7]{CMNR}, so our viewpoint on representability is a different one.

\subsection{Rank}

For a classical $\Delta$-matroid one has a rank function:
\[\rho(A)=|E|-\min \{|A \triangle F| : F \in \mathcal{F}\}.\]
Caution: This rank function does not specialize to the rank function of a matroid in the special case where the $\Delta$-matroid is a matroid. In the $q$-analogue, we can define the following.
\begin{definition} \label{defrank}
Let $\Delta=(E,\mathcal{F})$ be a $q$-$\Delta$-matroid. Then the rank fuction of $\Delta$ is defined by
\[ \rho(A)=\rho_\Delta(A)=n-\min \{\dim A +\dim F - 2 \dim (A \cap F) : F \in \mathcal{F}\}. \]
\end{definition}

Analogously to the case of classical $\Delta$-matroids, we then obtain: 
\begin{proposition} \label{rank}
For all subspaces $A\subseteq E$ we have:
\begin{enumerate}
\item \label{rank1} $A\in\mathcal{F}$ if and only if $\rho_\Delta(A)=n$.
\item \label{rank2} $\rho_{\Delta^*}(A^{\perp})=\rho_\Delta(A)$.
\end{enumerate}
\end{proposition}
\begin{proof}
For part (\ref{rank1}) we see that $A \in \mathcal{F}$ if and only if there is an $F \in \mathcal{F}$ such that 
\begin{align*}
\lefteqn{\dim A +\dim F - 2 \dim (A \cap F)} \\
&=(\dim A  -  \dim (A \cap F))+(\dim F -  \dim (A \cap F)) \\ & =0.
\end{align*}
This happens if and only if $A \cap F $ is equal to both $A$ and $F$, that is $A=F$.
 
For part (\ref{rank2}) we set $a=\dim A$ and $f=\dim F$. Then $n- \rho_{\Delta^*}(A^{\perp})$ is the minimum, taken over the $F \in \mathcal{F}$, of 
\begin{align*}
\lefteqn{(n-a) + (n-b) -2\dim(A^{\perp} \cap F^{\perp})} \\
&=(n-a) + (n-b) -2((n-a)+(n-b)-\dim(A^{\perp}+F^{\perp}) \\
&=a+f-2n+2\dim(A^{\perp}+F^{\perp}) \\
&=a+f-2n+2(n-\dim(A \cap F)) \\
&=a+f-2\dim (A \cap F).
\end{align*}
But the minimum of this, as $F$ varies over $\mathcal{F}$, is $n-\rho_\Delta(A)$.
Hence $\rho_\Delta(A)=\rho_{\Delta^*}(A^{\perp})$.
\end{proof}

Furthermore, we have the following relation between the ranks of the upper and lower $q$-matroid of a $q$-$\Delta$-matroid, and the rank of the $q$-$\Delta$-matroid itself.
\begin{proposition}
Let $\Delta=(E,\mathcal{F})$ be a $q$-$\Delta$-matroid. Let $\Delta_U$ and $\Delta_L$ be the upper and lower $q$-matroid of $\Delta$, respectively. In analogue with \cite[Lemma 5.38]{CMNR} we obtain
\[ r(\Delta_U)=\rho_\Delta(E)\text{ and } r(\Delta_L^*)=\rho_\Delta(\{0\}).\]
\end{proposition}
\begin{proof}
For the first half of the statement, we have
\begin{align*}
\rho_D(E) &=n -\min \{n +\dim F -2(\dim F\cap E)\} \\
&= n -\min \{n -\dim F\} \\
&= \max\{dim  F\} \\
& =r(\Delta_U).
\end{align*}
For the second half, we observe
\begin{align*}
\rho_D(\{0\}) &=n -\min \{0 +\dim F -2(\dim F\cap \emptyset)\} \\
&= n -\min \{\dim F\} \\
&= n - r(\Delta_L) \\
& =r(\Delta_L^*). \qedhere
\end{align*}
\end{proof}

In analogue with \cite{B2} we also have an alternative notion of rank in a $q$-$\Delta$-matroid.
 
\begin{definition} \label{defbirank}
Let $X,Y$ be orthogonal subspaces of $E$.
We define, for all such pairs of orthogonal spaces, its birank:
\[ \rho_b(X,Y)= \max \{\dim (F \cap X) + \dim (F^{\perp} \cap Y) : F \in \mathcal{F}\}. \]
\end{definition}

As one sees $X \in \mathcal{F}$ if and only $\rho_b(X,X^{\perp})=n$.

In \cite[Proposition 6.1]{B2} and its corollaries  one uses the corresponding birank for usual $\Delta$-matroids to give results for the rank functions of the upper and lower matroid for a $\Delta$-matroid.
We have not been able to understand enough of the arguments in \cite{B2} to be able to prove the corresponding results for $q$-$\Delta$-matroids.

\section*{Acknowledgement}

We are grateful to Steven D. Noble for help,  and useful answers, to questions about $\Delta$-matroids and $g$-matroids.

Trygve Johnsen was partially supported by the project Pure Mathematics in Norway, funded by the Trond Mohn Foundation, and also by the UiT Aurora project MASCOT. He also thanks the Department of Mathematics, IIT-Bombay, where he was a guest when a lot of his work with this article was done.

Michela Ceria works within the italian G.N.S.A.G.A. (Gruppo Nazionale per le Strutture Algebriche, Geometriche e le loro Applicazioni)

This work was supported  by the Italian Ministry of University and Research under the Programme “Department of Excellence” Legge 232/2016 (Grant No. CUP - D93C23000100001).

\end{document}